\documentclass[10pt,a4paper]{article}

\usepackage[a4paper]{geometry}
\usepackage{amssymb,latexsym,amsmath,amsfonts,amsthm}
\usepackage{graphicx}

\usepackage{epsfig}
\usepackage{comment}
\usepackage{subfigure}
\usepackage{graphicx}

\newcommand{\er}{\mathbb{R}}
\newcommand{\cee}{\mathbb{C}}
\newcommand{\enn}{\mathbb{N}}

\newcommand{\lam}{\lambda}

\newcommand{\res}{\mathrm{Res}}

\renewcommand{\sp}{\mathrm{sp}\ }

\newtheorem{theorem}{Theorem}[section]
\newtheorem{lemma}[theorem]{Lemma}
\newtheorem{proposition}[theorem]{Proposition}

\theoremstyle{definition}
\newtheorem{definition}[theorem]{Definition}
\newtheorem{example}[theorem]{Example}

\newtheorem{remark}[theorem]{Remark}

\numberwithin{equation}{section}

\title{An equilibrium problem for the limiting eigenvalue distribution of rational Toeplitz matrices}

\author{Steven Delvaux\footnotemark[1],\quad Maurice Duits\footnotemark[2]}
\date{}

\begin{document}

\maketitle
\renewcommand{\thefootnote}{\fnsymbol{footnote}}
\footnotetext[1]{Department of Mathematics, Katholieke Universiteit Leuven,
Celestijnenlaan 200B, B-3001 Leuven, Belgium. email:
steven.delvaux\symbol{'100}wis.kuleuven.be. \\
The first author is a Postdoctoral Fellow of the Fund for Scientific Research -
Flanders (Belgium).} \footnotetext[2]{Department of Mathematics, California
Institute of Technology, 1200 E. California Blvd., CA 91125, USA. email:
mduits\symbol{'100}caltech.edu.}

\begin{abstract}
We consider the asymptotic behavior of the eigenvalues of Toeplitz matrices with rational symbol
as the size of the matrix goes to infinity. Our main result is that the weak limit of the normalized eigenvalue
counting measure is a particular component of the unique solution to a vector equilibrium
problem. Moreover, we show that the other components describe the limiting behavior of certain generalized
eigenvalues. In this way, we generalize the recent results of Duits and Kuijlaars \cite{DK} for
banded Toeplitz matrices.



\end{abstract}

\section{Introduction}

For an integrable function $f$ on the complex unit circle $\{z\mid |z|=1\}$ the
\emph{Toeplitz matrix} $T_n(f)$ of size $n\times n$ is defined by
\begin{equation}\label{rationalToeplitzmx} T_n(f) =
\begin{pmatrix} f_{i-j}
\end{pmatrix}_{i,j=1}^{n} = \left(\begin{array}{lllll}
f_{0} & f_{-1} & f_{-2} & f_{-3} &  \ldots \\
f_{1} & f_{0} & f_{-1} & f_{-2} &  \ldots \\
f_{2} & f_{1} & f_{0} & f_{-1} &   \ldots \\
f_{3} & f_{2} & f_{1} & f_{0} &  \ldots \\
\vdots & \vdots & \vdots & \vdots  & \ddots
\end{array}\right)_{n\times n},
\end{equation}
where $f_k$ is the $k$th Fourier coefficient of $f$
\begin{align}\label{Fouriercoeff}
f_k=\frac{1}{2\pi} \int_0^{ 2\pi} f({\rm e}^{{\rm i}t}){\rm e}^{-{\rm i}k t} d
t.
\end{align}
The function $f$ is called the \emph{symbol} of $T_n(f)$. In this paper we will
be interested in symbols $f$ that are rational. That is, we assume that there
exist polynomials $A$, $B_1$ and $B_2$ such that
\begin{equation}\label{algebraiceq:0} f(z) = \frac{A(z)}{B_1(z)B_2(z)},
\end{equation}
where the  roots of $B_1$  (or $B_2$) lie inside (or outside) the unit circle.
Thus we do not allow $f$ to have poles on the unit circle. We take $A$ so that
it has no common roots with $B_1$ and $B_2$.

Note that if $B_1(z) = z^q$, $q\in\enn$, and $B_2(z)\equiv 1$, then
\eqref{algebraiceq:0} reduces to the Laurent polynomial
\begin{equation}\label{algebraiceq:banded} f(z) = \frac{A(z)}{z^q} = \sum_{k=-q}^p
f_k z^{k} \end{equation} where $$ p= \deg A - q. $$ Thus we have $f_k=0$ for
all $k>p$ and for all $k<-q$. The matrix $T_n(f)$ is then a \emph{banded
Toeplitz matrix}. The integers $p$ and $q$ in \eqref{algebraiceq:banded}
correspond to the outermost non-zero diagonals in the lower and upper
triangular part of this matrix, respectively. For a detailed discussion of
banded Toeplitz matrices see \cite{BG}.

We are interested in the asymptotic behavior of the eigenvalues of $T_n(f)$ as
$n\to \infty$. It is known that the eigenvalues accumulate on a particular
curve in the complex plane that we will  introduce shortly. Moreover, there
exists a measure on this curve describing the limiting distribution of the
eigenvalues. It was shown in \cite{DK} that for banded Toeplitz matrices this
limiting distribution is subject to an equilibrium problem that is naturally
constructed out of the symbol. The purpose of the present paper is to extend
this result to the case of rational symbols.\smallskip


Let us first review  some results on the asymptotic behavior of eigenvalues of
 rationally generated Toeplitz matrices.
Let $f$ be as  in \eqref{algebraiceq:0}  and $T_n(f)$ the associated Toeplitz matrix. Denote the
spectrum of $T_n(f)$ as
$$ \sp T_n(f) = \{\lam\in\cee\mid \det(T_n(f)-\lam I)=0\}.
$$
To describe the asymptotic behavior of the spectrum we introduce, as in
\cite{SchmidtSpitzer}, two different limiting sets
$$ \liminf_{n\to\infty} \sp T_n(f)
$$
consisting of all $\lam\in\cee$ for which there exists a sequence
$\{\lam_n\}_{n\in\enn}$, with $\lam_n\in\sp T_n(f)$ converging to $\lam$, and
the set
$$ \limsup_{n\to\infty} \sp T_n(f)
$$
consisting of all $\lam$ for which there exists a sequence
$\{\lam_n\}_{n\in\enn}$, with $\lam_n\in\sp T_n(f)$ having a subsequence
converging to $\lam$.

It turns out that these limiting sets can be described in terms of solutions to the equation
\begin{equation}\label{algebraiceq:1}
0=f(z)-\lam 
= \frac{A_{\lam}(z)}{B_1(z)B_2(z)}
\end{equation}
where \begin{equation}\label{def:flambda} A_{\lam}(z) := A(z)-\lam
B_1(z)B_2(z).
\end{equation} Following the analogy with \eqref{algebraiceq:banded},
we define
\begin{equation}\label{def:pq} q:=\deg B_1,\qquad p:=\max(\deg
A,\deg B_1\deg B_2)-q.\end{equation} To avoid trivial cases, in what follows we
always assume that $p,q\geq 1$, see e.g.\ \cite{Day2}. We also assume without
loss of generality that
\begin{equation}\label{gcd:condition}
\textrm{gcd}\{k\mid f_k\neq 0\} = 1,
\end{equation}
see \cite[Page 263]{BG}.

Note that $A_{\lam}(z)$ in \eqref{def:flambda} is a polynomial of degree $p+q$
in $z$, with each of its coefficients depending linearly on $\lam$. There can
be at most one value of $\lam\in\cee$ for which the leading coefficient
vanishes. For all other $\lam\in\cee$, the polynomial $A_{\lam}(z)$ has
precisely $p+q$ roots $z=z(\lam)$ (counting multiplicities) and we label them
by absolute value as
\begin{equation}\label{ordering:roots} 0\leq |z_1(\lam)|\leq |z_2(\lam)|\leq\ldots\leq
|z_{p+q}(\lam)|.
\end{equation}
In case where two or more subsequent roots in \eqref{ordering:roots} have the
same absolute value, we may arbitrarily label them so that
\eqref{ordering:roots} is satisfied. For the special value of $\lambda$ for
which the polynomial $A_\lambda$ has  less than $p+q$  roots, say $p+q-k$,  we
again order these roots  $z_1(\lambda),\ldots,z_{p+q-k}(\lambda)$ as in
\eqref{ordering:roots} and then we set
$z_{p+q-k+1}(\lambda)=\ldots=z_{p+q}(\lambda)=\infty$.

Define the curve
\begin{equation}\label{def:curves0} \Gamma_0 := \{\lam\in\cee \mid |z_{q}(\lam)| = |z_{q+1}(\lam)|
\}.
\end{equation}
The fact of the matter is
that
\begin{equation} \liminf_{n\to\infty} \sp T_n(f) = \limsup_{n\to\infty} \sp
T_n(f) = \Gamma_0.
\end{equation}
This result was first established by P. Schmidt and F. Spitzer
\cite{SchmidtSpitzer} in the banded Toeplitz case \eqref{algebraiceq:banded},
using a determinant identity by H. Widom \cite{Widom0}. The generalization to
the case of rational symbols \eqref{algebraiceq:0} is due to K.M. Day
\cite{Day2}, based on an extension \cite{Day1} of Widom's determinant
identity.\smallskip

Let $\nu_n$ be the counting measure
on the eigenvalues of $T_n(f)$
\begin{align} \label{eq:nun}
\nu_n=\frac{1}{n}\sum_{\lam \in \sp T_n(f)}\delta_\lambda
\end{align}
where $\delta_\lam$ is the Dirac measure at $\lambda$ and each eigenvalue is counted according to its multiplicity. It turns out
that the measures $\nu_n$ converge weakly to a measure  $\mu_0$ on $\Gamma_0$.

In the banded case \eqref{algebraiceq:banded} the measure $\mu_0$ is known to
be absolutely continuous, and an explicit expression for this measure was given
by I.I. Hirschman \cite{Hirschman}. An alternative representation of $\mu_0$
can be obtained by setting $k=0$ in \eqref{def:measuresk} below, cf.\
\cite{DK}. Further results about $\mu_0$ in the banded case can be found in
\cite{BG,DK,Hirschman,Ullman}.

For Toeplitz matrices with rational symbol \eqref{algebraiceq:0} the limiting
eigenvalue measure does not need to be absolutely continuous. Indeed, it was
shown by Day \cite{Day2} that this measure has an absolutely continuous part
together with at most two point masses.\smallskip


Finally, we turn to the results of \cite{DK}. Consider the general system of
curves
\begin{equation}\label{def:curves} \Gamma_k = \{\lam\in\cee \mid |z_{q+k}(\lam)| = |z_{q+k+1}(\lam)| \},
\end{equation}
for $k=-q+1,\ldots,p-1$. Each curve $\Gamma_k$ consists of finitely many
analytic arcs. We equip every analytic arc of $\Gamma_k$ with an orientation
and we define the $+$-side (or $-$-side) as the side on the left (or right) of
the arc when traversing the arc according to its orientation.

For $k=-q+1,\ldots,p-1$ we define the measure
\begin{equation}\label{def:measuresk}
d\mu_k(\lam) = \frac{1}{2\pi i}\sum_{j=1}^{q+k}\left(
\frac{z_{j+}'(\lam)}{z_{j+}(\lam)}-\frac{z_{j-}'(\lam)}{z_{j-}(\lam)}
\right)d\lam
\end{equation}
on the curve $\Gamma_k$. Here $d\lam$ denotes the complex line element on each
analytic arc of $\Gamma_k$, according to the chosen orientation of $\Gamma_k$.
Moreover, $z_{j+}(\lam)$ and $z_{j_-}(\lam)$ are the boundary values of
$z_j(\lambda)$ obtained from the $+$-side and $-$-side respectively of
$\Gamma_k$. These boundary values exist except for a finite number of points.
Note that \eqref{def:measuresk} is actually independent of the choice of the
orientation.

For the banded case it is shown in \cite{DK} that each $\mu_k$ is a finite
positive measure. Moreover, $\mu_0$ is the measure of Hirschman, that is,  the
limit of the normalized eigenvalue counting measures $\nu_n$ as given in
\eqref{eq:nun}.
The main observation in \cite{DK} is that the system of measures  $\{\mu_k\}_k$ together
uniquely minimizes an energy functional defined on the system of
curves $\{\Gamma_k\}_k$.

The purpose of this paper is to prove that also for  rational
symbols the measures $\{\mu_k\}_k$ minimize an energy functional, thus generalizing the results in \cite{DK}.
 The general
definition of the energy functional involves point sources that do not occur in
the banded Toeplitz case. This is related to the phenomenon that the limiting
eigenvalue distribution possibly has point masses for rationally generated
Toeplitz matrices, as mentioned before. We also emphasize that the $\mu_k$ are
absolutely continuous. It is to be understood that $\mu_0$ is the
\emph{absolutely continuous part} of the limiting eigenvalue distribution, with
the possible point masses removed. Our results will be stated in detail in the
next section.

\section{Statement of results}

\subsection{Auxiliary definitions}
\label{subsection:auxdef}

First we introduce some definitions that will be used in the statement of our
main theorems. For $k=0$ these definitions will be essentially the ones of Day
\cite{Day2}, but we will state the definitions for general values of
$k\in\{-q+1,\ldots,p-1\}$.

\begin{definition}\label{def:lambda12} Recall the notations \eqref{def:flambda}--\eqref{def:pq}.
Define the coefficients $a_k,b_k\in\cee$, $k=-q,\ldots,p$, by
\begin{align}\label{def:akbk}
A_\lambda(z)=A(z)-\lam B_1(z)B_2(z)=:\sum_{k=-q}^p
(a_k-\lambda b_k) z^{q+k}.
\end{align}
(Note the index shift.) Define $\lam_1,\lam_2\in\overline{\cee}:=\cee\cup\{\infty\}$ such that
$$ a_{-q}-b_{-q}\lam_1 = 0,\qquad\textrm{and } a_p-b_p\lam_2 = 0, $$
respectively. Define $k_1,k_2\in\{1,\ldots,p+q\}$ to be the maximal indices for
which
$$ a_{k}-b_{k}\lam_1 = 0,\qquad k=-q,\ldots,-q+k_1-1,
$$
and
$$ a_{k}-b_{k}\lam_2 = 0,\qquad k=p-k_2+1,\ldots,p,
$$
respectively. Here we make the convention that $a_k-b_k\lam_j = 0$ if $b_k=0$
and $\lam_j=\infty$, $j=1,2$.
\end{definition}

The numbers $\lam_1,\lam_2$ in Definition \ref{def:lambda12} are the unique
$\lam$-values for which the polynomial $A_\lam$  has some of its roots equal to $0$ (for
$\lambda_1$)  or to $\infty$ (for $\lambda_2$). In fact, the numbers $k_1$ and
$k_2$ are chosen such that $A_{\lam_1}$ has $k_1$ roots equal to zero, and
$A_{\lam_2}$ has $k_2$ roots at $\infty$. For all other values of $\lam$,
$A_{\lam}$ has precisely $p+q$ roots (counting multiplicities) which are all
non-zero and finite.

\begin{remark} \label{rem1}
By definition we have that $k_1,k_2\leq p+q$. The case where $k_1>p+q$ or
$k_2>p+q$ cannot occur since it would imply that all coefficients of
$A_{\lam}(z)$ are equal up to multiplication with a scalar. This would then imply that the
numerator and denominator in \eqref{algebraiceq:0} are equal up to a
scalar factor, contrary to our assumptions. Note also that it is possible to
have either $k_1=p+q$ or $k_2=p+q$, but not simultaneously. Indeed, in the
latter case we would have that $A_{\lambda}(z)=(a_{-q}-\lam b_{-q})+(a_p-\lam
b_p)z^{p+q}$ and hence $B_1(z)B_2(z)=b_{-q}+b_p z^{p+q}$. The latter implies
that either $p=0$ or $q=0$ which contradicts the assumption $p,q\geq 1$ made in
the introduction.
\end{remark}

\begin{definition}\label{def:mk}
For each $k=-q+1,\ldots,p-1$ define
\begin{equation}\label{index:m1} m_{1,k} =
\max\left(1-\frac{q+k}{k_1},0\right)\in [0,1),
\end{equation}
\begin{equation}\label{index:m2}
m_{2,k} = \max\left(1-\frac{p-k}{k_2},0\right)\in [0,1),
\end{equation}
and
\begin{equation}\label{index:m} m_k = 1-m_{1,k}-m_{2,k}.
\end{equation}
\end{definition}

The numbers $m_{1,k}$ and $m_{2,k}$ will be the weights
of certain point masses, see further. The quantity $m_k$ will be the total mass of the measure $\mu_k$ in
\eqref{def:measuresk}. Occasionally we will also consider
$m_{1,k}$, $m_{2,k}$ and $m_{k}$ for the indices $k=-q$ or $k=p$.

Note that the $m_k$ are strictly positive for all $k$. Indeed, from the
definition of $m_{1,k},m_{2,k},m_{k}$ and the fact that $k_1\leq p+q$ and
$k_2\leq p+q$  it is easy to check that $m_k\geq 0$ for $k=-q+1,\ldots,p-1$.
Moreover, if  $m_k=0$ for some $k$ then $k_1=k_2=p+q$. However, in
Remark~\ref{rem1} we observed that this is not possible.



\begin{example}\label{example:bandedcase}
Consider the banded case, i.e.\ $b_0=1$ and $b_k=0$ for all other $k$. In that
case we have $\lam_1=\lam_2=\infty$, $k_1=q$ and $k_2=p$. Then the numbers
$m_{1,k}$, $m_{2,k}$ and $m_{k}$ in Definition \ref{def:mk} are given in the
following table
\begin{equation}\label{tabel:bandedcase}
\left.\begin{array}{c|ccccccc}
k & -q+1 & \ldots & -1 & 0 & 1& \ldots & p-1 \\
\hline m_{1,k} & \frac{q-1}{q} & \ldots & \frac 1q & 0 & 0 & \ldots & 0 \\
m_{2,k} & 0 & \ldots & 0 & 0 & \frac 1p & \ldots & \frac{p-1}{p}\\
m_k & \frac 1q & \ldots   & \frac{q-1}{q} & 1 & \frac{p-1}{p} & \ldots & \frac
1p
\end{array}.\right.
\end{equation}
The last row of \eqref{tabel:bandedcase}  contains  the masses of the measures $\mu_k$ appearing in
\cite{DK}.
\end{example}


\begin{example} Here are two examples of possible behavior when $q=4$ and
$p=2$:
\begin{equation}\label{tabel:examplebehavior}
\left.\begin{array}{c|cccccc}
k & -3 & -2 &-1&0&1\\
\hline m_{1,k} & 2/3 & 1/3 & 0 & 0  & 0\\
m_{2,k} & 0   & 0   & 0   &  0 & 0\\
m_k & 1/3 & 2/3 & 1 & 1 & 1
\end{array},\right.\qquad
\left.\begin{array}{c|ccccccc}
k & -3 & -2 &-1&0&1\\
\hline
m_{1,k} & 4/5 & 3/5 & 2/5 & 1/5 & 0 \\
m_{2,k} & 0 & 0& 0& 1/3 & 2/3\\
m_k & 1/5 & 2/5 & 3/5 & 7/15 & 1/3
\end{array},\right.
\end{equation}
for the case where $(k_1,k_2)=(3,1)$ or $(5,3)$, respectively. This occurs
e.g.\ for the rational symbols
\[   \left\{\begin{array}{ll}
f(z) = \frac{1+z^2+2z^6}{1+z^2+z^3+z^5+z^6}& \textrm{(left table),}\\
f(z) =\frac{1+z^2+z^3+2z^5+2z^6}{1+z^2+z^3+z^5+z^6}& \textrm{(right
table).}\end{array}\right.\]
\end{example}

\subsection{The equilibrium problem}
\label{subsection:equil}

Below we will consider measures $\mu$ supported on contours in $\cee$. If the support is unbounded then we will assume that
$$ \int \log(1+|x|)\ d\mu(x)<\infty.
$$
For such a measure $\mu$ define its \emph{logarithmic energy} as
\begin{equation}\label{deflogarithmicenergy}
I(\mu) = \int_{}\int_{}\ \log\frac{1}{|x-y|} \ d\mu(x)\ d\mu(y).
\end{equation}
Similarly, for measures $\mu,\nu$ define their \emph{mutual energy} as
\begin{equation}\label{defmutualenergy}
I(\mu,\nu) = \int_{}\int_{}\ \log\frac{1}{|x-y|} \ d\mu(x)\ d\nu(y).
\end{equation}

\begin{definition} (Compare with \cite{DK})
We call a vector of measures $\vec{\nu} = (\nu_{-q+1},\ldots,\nu_{p-1})$
\emph{admissible} if $\nu_k$ has finite logarithmic energy, $\nu_k$ is
supported on $\Gamma_k$, and $\nu_k$ has total mass $\nu_k(\Gamma_k) = m_k$ for
every $k=-q+1,\ldots,p-1$, recall \eqref{index:m}.
\end{definition}

The \emph{energy functional} $J$ is defined by
\begin{multline}\label{energyfunctional} J(\vec{\nu}) = \sum_{k=-q+1}^{p-1}
I(\nu_k) - \sum_{k=-q+1}^{p-2} I(\nu_k,\nu_{k+1}) -
\frac{\chi_{\lam_1\neq\infty}\chi_{k_1<p+q}}{k_1}\int \log\frac{1}{|x-\lam_1|}\ d\nu_{-q+k_1}(x)\\
- \frac{\chi_{\lam_2\neq\infty}\chi_{k_2<p+q}}{k_2}\int
\log\frac{1}{|x-\lam_2|}\ d\nu_{p-k_2}(x).
\end{multline}
Here we define  $\chi_{\lam_1\neq\infty}$ to be $1$ when $\lam_1\neq\infty$ and
$0$ otherwise. The  quantities $\chi_{\lam_2\neq \infty}, \chi_{k_1<p+q}$ and
$\chi_{k_2<p+q}$ are similarly defined.

The \emph{equilibrium problem} is to minimize the energy functional
\eqref{energyfunctional} over all admissible vectors of positive measures
$\vec{\nu}$.

The equilibrium problem may be understood intuitively as follows. On each of
the curves $\Gamma_k$ one puts charged particles with total charge $m_k$.
Particles that lie on the same curve repel each other. The particles on two
consecutive curves attract each other, but with a strength that is only half as
strong as the repulsion on a single curve. Particles on different curves that
are non-consecutive do not interact with each other in a direct way.  In
addition, if $\lambda_1\neq\infty$ and $k_1<p+q$ then we have an external field
acting on the particles on the curve $\Gamma_{-q+k_1}$. Similarly, if
$\lambda_2\neq \infty$ and $k_2<p+q$ we have an external field acting on the
particles on $\Gamma_{p-k_2}$. The external fields come from point charges at
$\lam=\lam_1$ and $\lam=\lam_2$, respectively. The minus signs in
\eqref{energyfunctional} imply that these point charges are \emph{attractive}.
Such external fields are sometimes referred to as \lq sinks\rq. Note that there
are no external fields acting on the other measures $\nu_k$, $k\not\in
\{-q+k_1,p-k_2\}$.

Note that the external fields acting on the measures $\nu_{-q+k_1}$ and
$\nu_{p-k_2}$
do not occur in \cite{DK}. Indeed, in that case  we have
$\lam_1=\lam_2=\infty$ and hence  $\chi_{\lam_1\neq\infty}$ and $\chi_{\lam_2\neq\infty}$ in
\eqref{energyfunctional} vanish, see Example~\ref{example:bandedcase}.

For more information on equilibrium problems with external fields,
see~\cite{NS,Rans,SaffTotik}.

\begin{remark} 
In order for the above equilibrium problem to make sense, we need the energy
functional $J$ in \eqref{energyfunctional} to be bounded from below. A proof of
this boundedness will be given in Lemma~\ref{lemma:boundedbelow}. For the boundedness it is important to note that
\begin{equation}\label{builtin:security}\lam_1\not\in\Gamma_{-q+k_1} \textrm{  and  }
\lam_2\not\in\Gamma_{p-k_2},
\end{equation}
which follows immediately from the definitions of $k_1$ and $k_2$.  Hence the sinks  are not on the contours on which they are acting. 
\end{remark}

The following is our main theorem.

\begin{theorem}\label{theorem:main} Recall the notations in \eqref{def:pq}, \eqref{ordering:roots},
\eqref{def:curves} and \eqref{def:measuresk} and assume that $p,q\geq 1$. Then
\begin{itemize}
\item[(a)] The vector of measures $\vec{\mu} = (\mu_{k})_{k=-q+1}^{p-1}$ defined in \eqref{def:measuresk} is admissible.
\item[(b)] For each $k\in\{ -q+1,\ldots,p-1\}$ there exists a constant
$l_k\in\mathbb{R}$ such that \begin{multline} 2\int\log\frac{1}{|\lam-x|}\
d\mu_k(x) - \int\log\frac{1}{|\lam-x|}\
d\mu_{k+1}(x) - \int\log\frac{1}{|\lam-x|}\ d\mu_{k-1}(x) \\
-
\frac{\chi_{\lam_1\neq\infty}\chi_{k=-q+k_1}}{k_1}\log\frac{1}{|\lam-\lam_1|}
-
\frac{\chi_{\lam_2\neq\infty}\chi_{k=p-k_2}}{k_2}\log\frac{1}{|\lam-\lam_2|}=l_k,
\end{multline}
for $\lam\in\Gamma_k\setminus\{\lam_1,\lam_2\}$. Here we let $\mu_{-q}$ and
$\mu_{p}$ be the zero measures.
\item[(c)] $\vec{\mu} =(\mu_{k})_{k=-q+1}^{p-1}$ is the unique solution to the equilibrium
problem described above.
\end{itemize}
\end{theorem}

Theorem \ref{theorem:main} will be proved in Section
\ref{section:proofmaintheorem}.  Note that the equalities in Part~(b) are
precisely the Euler-Lagrange variational conditions of the equilibrium problem.
Part~(c) will then be a consequence of the convexity of the energy functional
$J$ and the fact that $J$ is bounded from below.

\subsection{The measures $\mu_k$ as limiting distributions of generalized
eigenvalues} \label{subsection:geneig}

It was proved in \cite{DK} that in case of banded Toeplitz matrices, the
measures $\mu_{k}$ for $k\neq 0$ also have an interpretation of being the
limiting measures for certain generalized eigenvalues. For the rationally
generated Toeplitz matrices such a result remains valid.

\begin{definition} \label{def:countingmeasuresrational}
For $k = -q+1,\ldots,p-1$ and $n\geq 1$ we define the polynomial
${P}_{k,n}$ by
$$ {P}_{k,n}(\lam) = \det T_n(z^{-k}(f(z)-\lam))
$$
and we define the \emph{$k$th generalized spectrum} of $T_n(f)$ by
$$ \mathrm{sp}_k T_n(f) = \{\lam\in\cee \mid {P}_{k,n}(\lam)=0\}.
$$
Finally, we define $\mu_{k,n}$ as the normalized zero counting measure of
$\mathrm{sp}_k T_n(f)$
$$  {\mu}_{k,n} = \frac 1n \sum_{\lam\in\mathrm{sp}_k T_n(f)} \delta_{\lam},
$$
where in the sum each $\lam$ is counted according to its multiplicity as a zero
of ${P}_{k,n}$.
\end{definition}

The Toeplitz matrix $T_n(z^{-k}(f-\lam))$ in
Definition~\ref{def:countingmeasuresrational} may be interpreted as a shifted
version of $T_n(f-\lam)=T_n(f)-\lam I$. For example if $k=2$ then by
\eqref{Fouriercoeff} we have
\begin{equation}\label{rationalToeplitzmxshifted} T_n(z^{-2}(f-\lam)) = \begin{pmatrix}
f_{2} & f_{1} & f_{0}-\lam & f_{-1} & f_{-2} & \ldots \\
f_{3} & f_{2} & f_{1} & f_{0}-\lam & f_{-1}  & \ldots \\
f_{4} & f_{3} & f_{2} & f_{1} & f_{0}-\lam & \ldots \\
f_{5} & f_{4} & f_{3} & f_{2} & f_{1} & \ldots \\
f_{6} & f_{5} & f_{4} & f_{3} & f_{2} & \ldots \\
\vdots & \vdots & \vdots & \vdots & \vdots & \ddots
\end{pmatrix}_{n\times n}.\end{equation}

We will show that for each $k$ the sequence $\{\mu_{k,n}\}_n$ has a limit.
Moreover, the limiting measure will have point masses at $\lam=\lam_1$ (if
$\lam_1\neq\infty$) and $\lam=\lam_2$ (if $\lam_2\neq\infty$) with weights at
least $m_{1,k}$ and $m_{2,k}$, respectively. On the other hand, if
$\lam_1=\infty$ or $\lam_2=\infty$ then the total mass of the limiting measure
is reduced with at least $m_{1,k}$ or $m_{2,k}$ respectively. These facts can
already be seen at the level of the finite-$n$ measures $\mu_{k,n}$ as the
following proposition shows.

\begin{proposition}\label{prop:combinatorial2}  Let $k\in\{-q+1,\ldots,p-1\}$.
Then the polynomial $P_{k,n}(\lam)$ satisfies the following properties.
\begin{itemize}
\item[(a)] 
$\left\{\begin{array}{ll} \textrm{$P_{k,n}(\lam)$ is divisible by
$(\lam-\lam_1)^{m_{1,k}n-c}$,} & \textrm{if $\lam_1\in\cee$},\\
\textrm{$P_{k,n}(\lam)$ has degree at most $(1-m_{1,k})n+c$,} & \textrm{if
$\lam_1=\infty$,}
\end{array}\right.$
\item[(b)] 
$\left\{\begin{array}{ll} \textrm{$P_{k,n}(\lam)$ is divisible by
$(\lam-\lam_2)^{m_{2,k}n-c}$,} & \textrm{if $\lam_2\in\cee$},\\
\textrm{$P_{k,n}(\lam)$ has degree at most $(1-m_{2,k})n+c$,} & \textrm{if
$\lam_2=\infty$,}
\end{array}\right.$
\end{itemize}
where $c\geq 0$ is a constant depending only on the symbol $f$. 

Denote by $Q_{k,n}(\lam)$ the quotient polynomial obtained from $P_{k,n}(\lam)$
by removing all its factors $(\lam-\lam_1)$ (if $\lam_1\in\cee$) and
$(\lam-\lam_2)$ (if $\lam_2\in\cee$). Then we have that
\begin{itemize}\item[(c)] \textrm{$Q_{k,n}(\lam)$ has degree at most
$m_{k}n+2c$.}
\end{itemize}
\end{proposition}

Proposition~\ref{prop:combinatorial2} will be proved in
Section~\ref{subsection:proofpropcomb}.

Since the measure \eqref{def:measuresk} is absolutely continuous, the best one
can hope for is $\mu_k$ being the \emph{absolutely continuous part} of the
limiting $k$th generalized eigenvalues distribution. This means that the
possible point masses at $\lam_1$ and $\lam_2$ should be stripped out in
$\mu_k$. This turns out to be indeed the case.

\begin{theorem}\label{theorem:geneigrat} Let $k\in\{-q+1,\ldots,p-1\}$. Then
$$ \liminf_{n\to\infty} \mathrm{sp}_k T_n(f ) = \limsup_{n\to\infty} \mathrm{sp}_k
T_n(f) = \Gamma_k,
$$
and
\begin{equation}\label{weakconvergence}
\lim_{n\to\infty} \int_{\cee} \phi(\lam)\ d\mu_{k,n}(\lam)= \int_{\cee}
\phi(\lam)\ d\mu_{k}(\lam) + \chi_{\lam_1\neq\infty}m_{1,k}\phi(\lam_1) +
\chi_{\lam_2\neq\infty}m_{2,k}\phi(\lam_2)
\end{equation}
holds for every bounded continuous function $\phi$ on $\cee$.
\end{theorem}

From \eqref{weakconvergence} we see that $m_{1,k}$ and $m_{2,k}$ are the
weights of the point masses at $\lambda_1$ and $\lambda_2$ in the limiting
$k$th generalized eigenvalues distribution, if present.

\subsection{Organization of the rest of the paper}

The rest of this paper is organized as follows.
Section~\ref{section:proofmaintheorem} contains the proof of
Theorem~\ref{theorem:main}. In Section~\ref{section:prooftheoremgeneig} we
prove Proposition~\ref{prop:combinatorial2} and
Theorem~\ref{theorem:geneigrat}. Most of the proofs are inspired by the proofs
given in \cite{DK} for the corresponding statements in the banded case, hence
we will often refer to that paper. Finally, some illustrations of our results
are given in Section \ref{section:examples}.

\section{Proof of Theorem \ref{theorem:main}} \label{section:proofmaintheorem}

\subsection{Proof of Theorem \ref{theorem:main}(a)--(b)}

In this section we will prove Theorem \ref{theorem:main}(a)--(b). First we
recall some elementary definitions and properties involving the algebraic
equation $A_{\lambda}(z) = 0$.

\begin{definition}(\cite[Section 11.2]{BG},\cite{DK})
A point $\lam\in\cee$ is called a \emph{branch point} if $A_{\lambda}(z) = 0$
has a multiple root. A point $\lam\in\Gamma_k$ is an \emph{exceptional point}
of $\Gamma_k$ if $\lam$ is a branch point, or if there is no open neighborhood
$U$ of $\lambda$ such that $\Gamma_k\cap U$ is an analytic arc starting and
terminating on $\partial U$.
\end{definition}

\begin{proposition}\label{prop:structureGammak}
Let $k\in\{-q+1,\ldots,p-1\}$. Then the set $\Gamma_k$ in \eqref{def:curves} is
the disjoint union of a finite number of open analytic arcs and a finite number
of exceptional points. The set $\Gamma_k$ has no isolated points.
\end{proposition}

The proof of this proposition is similar as in \cite[Theorem
11.9]{BG},\cite{DK},\cite{SchmidtSpitzer}. The condition \eqref{gcd:condition}
is needed to ensure that the $\Gamma_k$ are proper curves, i.e., they are
1-dimensional subsets of $\cee$.

A major role is played by the functions $w_k$ which, for $k=-q+1,\ldots,p-1$
are defined by
\begin{equation}\label{def:wk}
w_k(\lam) = \prod_{j=1}^{q+k} z_j(\lam),\quad \lam\in\cee\setminus\Gamma_k.
\end{equation}
The function $w_k$ is analytic in $\cee\setminus\Gamma_k$. Occasionally we will
also consider $w_k$ for the indices $k=-q$ or $k=p$.

Note that \eqref{def:measuresk} may be written alternatively as
$$ d\mu_k(\lam) = \frac{1}{2\pi i}\left( \frac{w_{k+}'(\lam)}{w_{k+}(\lam)}-\frac{w_{k-}'(\lam)}{w_{k-}(\lam)}
\right)d\lam.
$$
To discuss the integrability of this measure, we will need the asymptotic
behavior of $\frac{w_k'(\lam)}{w_k(\lam)}$. The relevant facts are listed in
the following proposition.

\begin{proposition}\label{prop:threeproperties} Let $k\in\{-q+1,\ldots,p-1\}$ and recall the notations in
Section~\ref{subsection:auxdef}. Then the following statements hold
\begin{enumerate}
\item[(a)] For any $\lam_0\in\overline{\cee}\setminus\{\lam_1,\lam_2\}$,
there exists an $l\in\enn:=\{1,2,3,\ldots\}$ such that
$$
\frac{w_k'(\lam)}{w_k(\lam)} = \left\{\begin{array}{ll}
O((\lam-\lam_0)^{-1+1/l}), & \textrm{ if
}\lam_0\in\cee\\
O(\lam^{-1-1/l}), & \textrm{ if }\lam_0=\infty\end{array}\right.
$$
as $\lam\to\lam_0$ with $\lam\in\cee\setminus\Gamma_k$. We have $l=1$ unless
$\lam_0$ is a branch point.

\item[(b)] Assume $\lam_1\neq\lam_2$. Then near the point $\lam_1$ there exists an $l\in\enn$ such that
$$
\frac{w_k'(\lam)}{w_k(\lam)} = \left\{\begin{array}{ll}
\frac{1-m_{1,k}}{\lam-\lam_1} + O((\lam-\lam_1)^{-1+1/l}), & \textrm{ if
}\lam_1\in\cee\vspace{1mm} \\ -\frac{1-m_{1,k}}{\lam} + O(\lam^{-1-1/l}), &
\textrm{ if }\lam_1=\infty
\end{array}\right.
$$
as $\lam\to\lam_1$ with $\lam\in\cee\setminus\Gamma_k$.

\item[(c)] Assume $\lam_1\neq\lam_2$. Then near the point $\lam_2$ there exists an $l\in\enn$ such that
$$
\frac{w_k'(\lam)}{w_k(\lam)} = \left\{\begin{array}{ll}
\frac{-m_{2,k}}{\lam-\lam_2} + O((\lam-\lam_2)^{-1+1/l}), & \textrm{ if
}\lam_2\in\cee\vspace{1mm} \\ \frac{m_{2,k}}{\lam} + O(\lam^{-1-1/l}), &
\textrm{ if }\lam_2=\infty
\end{array}\right.
$$
as $\lam\to\lam_2$ with $\lam\in\cee\setminus\Gamma_k$.

\item[(d)] If $\lam_1=\lam_2$, then
$$
\frac{w_k'(\lam)}{w_k(\lam)} = \left\{\begin{array}{ll}
\frac{m_{k}}{\lam-\lam_1} + O((\lam-\lam_1)^{-1+1/l}), & \textrm{ if
}\lam_1\in\cee\vspace{1mm} \\ \frac{-m_{k}}{\lam} + O(\lam^{-1-1/l}), &
\textrm{ if }\lam_1=\infty
\end{array}\right.
$$
as $\lam\to\lam_1=\lam_2$ with $\lam\in\cee\setminus\Gamma_k$.
\end{enumerate}
\end{proposition}

\begin{proof} First we make some general observations. For any
$\lam\in\overline{\cee}\setminus\{\lam_1,\lam_2\}$ the polynomial
$A_{\lam}(z)=0$ has roots $z_j(\lam)$, $j=1,\ldots,p+q$, all of which are
finite and non-zero (although some of the roots might occur with higher
multiplicity). Zero roots or infinite roots can only occur if
$\lam\in\{\lam_1,\lam_2\}$. As $\lam\to\lam_1$, then the $k_1$ smallest roots
$z_j(\lam)$ tend to zero like a power $(\lam-\lam_1)^{1/k_1}$, while the other
roots converge to non-zero constants. Similarly, as $\lam\to\lam_2$, then the
$k_2$ largest roots $z_j(\lam)$ tend to infinity like a power
$(\lam-\lam_2)^{-1/k_2}$, while the other roots converge to non-zero constants.
These facts all follow from Definition \ref{def:lambda12}.

To prove the first equality of Part (a), fix
$\lam_0\in\cee\setminus\{\lam_1,\lam_2\}$ and denote with $z_j=z_j(\lam_0)$,
$j=1,\ldots,p+q$, the roots of $A_{\lam_0}(z)$; by the discussion in the
previous paragraph we have $z_j\in\cee\setminus\{0\}$ for each $j$. Pick one of
the roots $z_j$ which has the highest multiplicity $l$. Writing
$$A_{\lam}(z) = (z-z_j)^l P(z)+(\lam-\lam_0)Q(z)$$ where $P(z)$ and $Q(z)$ are
polynomials with $P(z_j)\neq 0$ and $Q(z_j)\neq 0$, it follows that
$$z_j(\lam) = z_j+c(\lam-\lam_0)^{1/l},\qquad \lam\to\lam_0, $$ for some constant $c$.
By taking the logarithmic derivative we obtain
$$
\frac{z_j'(\lam)}{z_j(\lam)} = O((\lam-\lam_0)^{-1+1/l}),\qquad \lam\to\lam_0.
$$
The first equality in Part (a) follows from this and the fact that
$$ \frac{w_k'(\lam)}{w_k(\lam)} =
\sum_{j=1}^{q+k}\frac{z_j'(\lam)}{z_j(\lam)}.
$$

The second equality in Part (a) (for the case $\lam_0=\infty$) is proved in a
similar way, this time using a decomposition
$$A_{\lam}(z) = (z-z_j)^l \lam P(z)+Q(z)$$ where again $P(z_j)\neq 0$ and
$Q(z_j)\neq 0$.

To prove Part (b), first assume that $\lam_1\neq\infty$. From the discussion in
the first paragraph of this proof we obtain
$$
\frac{z_j'(\lam)}{z_j(\lam)} = \frac{1}{k_1(\lam-\lam_1)} +
O((\lam-\lam_1)^{-1+1/k_1}),\quad \lam\to\lam_1,
$$
for all $j=1,\ldots,k_1$, while
$$
\frac{z_j'(\lam)}{z_j(\lam)} = O((\lam-\lam_1)^{-1+1/l}),\quad \lam\to\lam_1,
$$
for $j=k_1+1,\ldots,p+q$ and a suitable $l\in\enn$. Hence
\begin{eqnarray*}
\frac{w_k'(\lam)}{w_k(\lam)} &=& \sum_{j=1}^{q+k}\frac{z_j'(\lam)}{z_j(\lam)} \\
   &=& \min\left(\frac{q+k}{k_1},1\right)\frac{1}{\lam-\lam_1} +
   O((\lam-\lam_1)^{-1+1/\tilde l}) \\
   &=& \frac{1-m_{1,k}}{\lam-\lam_1} + O((\lam-\lam_1)^{-1+1/\tilde l})
\end{eqnarray*}
by virtue of \eqref{index:m1}, with $\tilde l = \max(l,k_1)$. Similarly one can
prove the case where $\lam_1=\infty$. The proofs of Parts (c) and (d) are
similar as well. \end{proof}

\begin{proposition}\label{prop:mupositivemass}
For each $k$ we have that $\mu_k$ in \eqref{def:measuresk} is a positive
measure on $\Gamma_k$ with total mass $\mu_k(\Gamma_k) = m_k$ as defined in
\eqref{index:m}.
\end{proposition}

\begin{proof} First we prove that the
density \eqref{def:measuresk} is locally integrable around the points $\lam_1$,
$\lam_2$, $\infty$ (at least those of them which lie on the curve $\Gamma_k$).
For $\infty$ this follows from the second equality in Proposition
\ref{prop:threeproperties}(a). For $\lam_1$ this follows from Proposition
\ref{prop:threeproperties}(b) and taking into account that the
$1/(\lam-\lam_1)$ terms at the $+$-side and $-$-side in \eqref{def:measuresk}
cancel; a similar argument holds for the point $\lam_2$.

The fact that the measure $\mu_k$ is positive follows from a Cauchy-Riemann
argument as in \cite[Proof of Proposition 4.1]{DK}.

Finally, the statement that $\mu_k(\Gamma_k) = m_k$ follows from a contour
deformation argument as in \cite[Proof of Proposition 4.1]{DK}. More precisely,
we have
\begin{align}\nonumber \mu_k(\Gamma_k) &:= \frac{1}{2\pi i}\int_{\Gamma_k}\left(
\frac{w_{k+}'(\lam)}{w_{k+}(\lam)}-\frac{w_{k-}'(\lam)}{w_{k-}(\lam)}
\right)d\lam \\
\label{residues1}&= \frac{1}{2\pi i}\int_{\mathcal C}
\frac{w_{k}'(\lam)}{w_{k}(\lam)}\ d\lam
+\chi_{\lam_1\neq\infty}\res\left(\frac{w_{k}'(\lam)}{w_{k}(\lam)},\lam=\lam_1\right)
+\chi_{\lam_2\neq\infty}\res\left(\frac{w_{k}'(\lam)}{w_{k}(\lam)},\lam=\lam_2\right),
\end{align}
where $\mathcal C$ is a clockwise oriented contour surrounding $\Gamma_k$ and
those points $\lam_1$ and $\lam_2$ which are finite, and where $\res(h,\lam)$
denotes the residue of $h$ at $\lam$. Note that \eqref{residues1} is valid even
when some of the $\lam_j$ lie on the curve $\Gamma_k$, $j=1,2$, thanks to the
local integrability of $\mu_k$ around these points. Applying the residue
theorem once again, this time for the exterior domain of $\mathcal C$, we then
find for the first term in \eqref{residues1} that
\begin{equation}\label{residues2} \frac{1}{2\pi i}\int_{\mathcal C}
\frac{w_{k}'(\lam)}{w_{k}(\lam)}\ d\lam =
-\res\left(\frac{w_{k}'(\lam)}{w_{k}(\lam)},\lam=\infty\right).
\end{equation}
The fact that $\mu_k(\Gamma_k) = m_k$ then follows from
\eqref{residues1}--\eqref{residues2} and the residue expressions in
Proposition~\ref{prop:threeproperties}; recall also \eqref{index:m}.
\end{proof}

\begin{proposition}\label{prop:cauchytransformcontour}
For each $k$ we have that \begin{equation}\label{cauchytransform:hulp}
\int\frac{1}{\lam-x}\ d\mu_k(x) =
-\frac{w_k'(\lam)}{w_k(\lam)}+\chi_{\lam_1\neq\infty}\frac{1-m_{1,k}}{\lam-\lam_1}
+\chi_{\lam_2\neq\infty}\frac{-m_{2,k}}{\lam-\lam_2},\quad\textrm{if
}\lam\in\cee\setminus\Gamma_k\end{equation} and
\begin{multline}\label{cauchytransform:muk} \int\log|\lam-x|\
d\mu_k(x) =
-\log|w_k(\lam)|+\chi_{\lam_1\neq\infty}(1-m_{1,k})\log|\lam-\lam_1|\\
-\chi_{\lam_2\neq\infty}m_{2,k}\log|\lam-\lam_2|+\alpha_k,\quad\textrm{if
}\lam\in\cee
\end{multline}
for a suitable constant $\alpha_k$.
\end{proposition}

\begin{remark} The quantities in the above proposition are only well-defined
if $\lam\neq\lam_j$, $j=1,2$. However, one easily checks that $\lam_1$ and
$\lam_2$ are removable singularities for the right hand sides of both
\eqref{cauchytransform:hulp} and \eqref{cauchytransform:muk}, due to the
continuity of the corresponding left hand sides.
\end{remark}

\begin{proof}[Proof of Proposition~\ref{prop:cauchytransformcontour}]
The proof of \eqref{cauchytransform:hulp} follows by contour deformation in a
similar way as in the proof of Proposition~\ref{prop:mupositivemass}. The
relevant expression is now
\begin{multline*}\int_{\Gamma_k}\frac{1}{\lam-x}\ d\mu_k(x) := \frac{1}{2\pi i}\int_{\Gamma_k}\frac{1}{\lam-x}
\left(\frac{w_{k+}'(x)}{w_{k+}(x)}-\frac{w_{k-}'(x)}{w_{k-}(x)} \right)dx
= \frac{1}{2\pi i}\int_{\mathcal C} \frac{1}{\lam-x}\frac{w_{k}'(x)}{w_{k}(x)}\
dx
\\ -\frac{w_k'(\lam)}{w_k(\lam)} + 
\chi_{\lam_1\neq\infty}\frac{1}{\lam-\lam_1}\res\left(\frac{w_{k}'(x)}{w_{k}(x)},x=\lam_1\right)
+\chi_{\lam_2\neq\infty}\frac{1}{\lam-\lam_2}\res\left(\frac{w_{k}'(x)}{w_{k}(x)},x=\lam_2\right),
\end{multline*}
where $\mathcal C$ is a clockwise oriented contour surrounding $\Gamma_k$, the
point $\lam$, and those points $\lam_1$ and $\lam_2$ which are finite. Now the
integrand in the integral over $\mathcal C$ has zero residue at infinity and
therefore this integral vanishes. Using the residue expressions in
Proposition~\ref{prop:threeproperties} one then arrives at the right hand side
of \eqref{cauchytransform:hulp}.
Finally, the proof of \eqref{cauchytransform:muk} then follows by integrating
\eqref{cauchytransform:hulp}, see also \cite{DK}.
\end{proof}

Now we are ready to prove Theorem \ref{theorem:main}(a)--(b).

\begin{proof}[Proof of Theorem \ref{theorem:main}(a)] Taking into account Proposition \ref{prop:mupositivemass},
it suffices to show that the logarithmic energy $I(\mu_k)$ is bounded for each
$k\in\{-q+1,\ldots,p-1\}$. The latter follows by integrating
\eqref{cauchytransform:muk} over $\mu_k(\lam)$. Then the left hand side becomes
$-I(\mu_k)$, so it suffices to show that each of the four terms in the right
hand side is bounded. For the fourth term this is evident since $\mu_k$ has
finite mass. For the two middle terms this follows from our earlier observation
that $\mu_k$ is integrable around $\lam_1$ and $\lam_2$ (assuming they are on
the curve $\Gamma_k$), which is still true when multiplying with the
logarithmic singularities $\log|\lam-\lam_1|$ and $\log|\lam-\lam_2|$. A
similar argument holds for the first term.
\end{proof}

\begin{proof}[Proof of Theorem \ref{theorem:main}(b)]
The proof of Part (b) follows from \eqref{cauchytransform:muk} and the
auxiliary results
\begin{multline*}-\log|w_{k+1}(\lam)|+2\log|w_k(\lam)|-\log|w_{k-1}(\lam)| =
\log\left|\frac{\prod_{j=1}^{k}z_j(\lam)^2}{\prod_{j=1}^{k+1}z_j(\lam)
\prod_{j=1}^{k-1}z_j(\lam)}\right|= \log\left|\frac{z_k(\lam)}{z_{k+1}(\lam)}\right| = 0
\end{multline*}
for $\lam\in\Gamma_k\setminus\{\lam_1,\lam_2\}$, and
\begin{equation}\label{summk:arithmetic1} -m_{1,k+1}+2m_{1,k}-m_{1,k-1} =
\left\{\begin{array}{ll}
-1/k_1 & k = -q+k_1 \\
0 & k\in\{-q+1,\ldots,p-1\}\setminus\{-q+k_1\}
\end{array},\right.\end{equation}
\begin{equation}\label{summk:arithmetic2} -m_{2,k+1}+2m_{2,k}-m_{2,k-1} =
\left\{\begin{array}{ll}
-1/k_2 & k = p-k_2 \\
0 & k\in\{-q+1,\ldots,p-1\}\setminus\{p-k_2\}
\end{array}.\right.\end{equation}
Here the boundary terms $m_{1,k}$, $m_{2,k}$ for $k=-q$ or $k=p$ are defined by
the usual formulae \eqref{index:m1}--\eqref{index:m2}. These considerations
imply the desired result for $\lam\in\Gamma_k\setminus\{\lam_1,\lam_2\}$; the
cases $\lam=\lam_1$ and $\lam=\lam_2$ then follow by continuity.
\end{proof}

\subsection{Proof of Theorem \ref{theorem:main}(c)}

To prove Theorem \ref{theorem:main}(c) we rewrite \eqref{energyfunctional} in
the following way, compare with \cite[Eq.~(2.12)]{DK}:
\begin{multline}\label{energy:alternativerepr} J(\vec{\nu}) = \left(\sum_{k=-q+1}^{p-2} \frac{m_k m_{k+1}}{2}
I\left(\frac{\nu_k}{m_k} - \frac{\nu_{k+1}}{m_{k+1}}\right)\right)
+\frac{1}{2k_1 m_{k_1}} I(\nu_{-q+k_1})+\frac{1}{2k_2 m_{k_2}} I(\nu_{p-k_2})\\
- \frac{\chi_{\lam_1\neq\infty}}{k_1}\int \log\frac{1}{|x-\lam_1|}\
d\nu_{-q+k_1}(x) - \frac{\chi_{\lam_2\neq\infty}}{k_2}\int
\log\frac{1}{|x-\lam_2|}\ d\nu_{p-k_2}(x).
\end{multline}
We leave it to the reader to check the correctness of this identity; note that
the calculation makes use of the auxiliary result
\begin{equation*} -m_{k+1}+2m_{k}-m_{k-1} =
\left\{\begin{array}{ll}
1/k_1 & k = -q+k_1 \neq p-k_2 \\
1/k_2 & k = p-k_2 \neq -q+k_1\\
1/k_1+1/k_2 & k = -q+k_1 = p-k_2 \\
0 & k\in\{-q+1,\ldots,p-1\}\setminus\{-q+k_1,p-k_2\}
\end{array}\right.
\end{equation*}
for $k\in\{-q+1,\ldots,p-1\}$, which follows from \eqref{index:m} and
\eqref{summk:arithmetic1}--\eqref{summk:arithmetic2}. Here we recall the
boundary values $m_{-q}=m_p=0$.

We also invoke the fact that \begin{equation}\label{Simeonov:estimate}
I(\nu_1-\nu_2)\geq 0,
\end{equation}
whenever $\nu_1$ and $\nu_2$ are positive measures with
$\nu_1(\cee)=\nu_2(\cee)\leq \infty$. This is a well-known result if $\nu_1$
and $\nu_2$ have bounded support \cite{SaffTotik}. If the support is unbounded
this is a recent result of Simeonov \cite{Simeonov}.

\begin{lemma}\label{lemma:boundedbelow}
The energy functional \eqref{energyfunctional} is bounded from below on the set
of admissible vectors of measures $\vec{\nu}$.
\end{lemma}

\begin{proof} From \eqref{energy:alternativerepr}--\eqref{Simeonov:estimate} we see that
in order to show that the energy functional $J(\vec{\nu})$ is bounded from
below, it is sufficient to show that
\begin{equation}\label{energy:alternativerepr2}
\frac{1}{2k_1 m_{k_1}} I(\nu_{-q+k_1}) -
\frac{\chi_{\lam_1\neq\infty}}{k_1}\int \log\frac{1}{|x-\lam_1|}\
d\nu_{-q+k_1}(x)
\end{equation}
and
$$
\frac{1}{2k_2 m_{k_2}} I(\nu_{p-k_2}) -
\frac{\chi_{\lam_2\neq\infty}}{k_2}\int \log\frac{1}{|x-\lam_2|}\
d\nu_{p-k_2}(x)
$$
are both bounded from below on the set of admissible vectors of measures
$\vec{\nu}$. Let us check this for the first term
\eqref{energy:alternativerepr2}. We will use that
$$\lam_1\not\in\Gamma_{-q+k_1},$$
a fact already observed in \eqref{builtin:security}, which follows immediately
from the definition of $k_1$. Now we distinguish between two cases. The first
case is when $\lam_1=\infty$. Then the second term in
\eqref{energy:alternativerepr2} drops out while on the other hand
$\infty=\lam_1\not\in\Gamma_{-q+k_1}$, so the contour $\Gamma_{-q+k_1}$ is
bounded and therefore the first term in \eqref{energy:alternativerepr2} is
bounded from below as well.

The second case is when $\lam_1\neq\infty$. Then standard arguments from
potential theory show that the expression \eqref{energy:alternativerepr2} is
minimized precisely when $\nu_{-q+k_1}$ is the \emph{balayage} of the Dirac
point mass at $\lam_1$ onto the curve $\Gamma_{-q+k_1}$, and in particular this
expression is bounded from below \cite[Chapter 2]{SaffTotik}.
\end{proof}

\begin{remark} The above proof goes through because the constant factor in front of the first term
in \eqref{energy:alternativerepr2} is precisely $1/2$. If this constant factor
is different from $1/2$ then the connection with balayage measures breaks down,
and in fact if the constant is larger than $1/2$ then the energy functional is
not bounded from below anymore.
\end{remark}

\begin{proof}[Proof of Theorem \ref{theorem:main}(c)]
Assume that $\vec{\mu}$ is a vector of admissible measures satisfying the
equalities in Theorem \ref{theorem:main}(b), and let $\vec{\nu}$ be any
admissible vector of measures. We need to prove that $J(\vec{\nu})\geq
J(\vec{\mu})$ with equality if and only if $\vec{\nu}=\vec{\mu}$. Note that the
equalities in Theorem \ref{theorem:main}(b) are precisely the Euler-Lagrange
variational conditions of the equilibrium problem. The result then follows from
the fact that the energy functional $J$ is convex and bounded from below. More
precisely, one can use exactly the same argument as in \cite[Proofs of Lemma
2.3 and Theorem 2.3(c)]{DK}, taking into account
\eqref{energy:alternativerepr}--\eqref{Simeonov:estimate}. There are some
modifications induced by the external fields, but this does not lead to
problems since the latter act in a linear way on the measures.
\end{proof}

\section{Proofs of Proposition \ref{prop:combinatorial2} and Theorem \ref{theorem:geneigrat}}
\label{section:prooftheoremgeneig}

\subsection{Proof of Proposition \ref{prop:combinatorial2}}
\label{subsection:proofpropcomb}

The proof of Proposition \ref{prop:combinatorial2} is based on the reduction of
a rationally generated Toeplitz matrix into banded form, which will then allow
us to follow the proof in \cite[Proof of Prop.~2.5]{DK}. Let us recall from
\eqref{algebraiceq:1} that $$ f(z)-\lam  = \frac{A_{\lam}(z)}{B_1(z)B_2(z)},
$$ where the numerator $A_{\lam}(z)$ is a polynomial in $z$.
Then we claim that for any $k\in\{-q+1,\ldots,p-1\}$ and for any $n$
sufficiently large, the rationally generated Toeplitz matrix with symbol
$z^{-k}(f(z)-\lam)$ can be reduced into banded form
by the factorization
\begin{equation}\label{wiener}
L_nT_n\left(z^{-k}(f(z)-\lam)\right)R_n =
T_n\left(z^{-q-k}A_{\lam}(z)\right)+\begin{pmatrix} C & 0 \\ 0 & 0
\end{pmatrix}_{n\times n},
\end{equation}
where
\begin{equation*}
L_n= T_n(B_2(z)),\qquad R_n:=T_n\left(z^{-q}B_1(z)\right),
\end{equation*}
are non-singular lower and upper triangular Toeplitz matrices respectively. The
middle factor in the left hand side of \eqref{wiener} is our rationally
generated Toeplitz matrix of interest, and \eqref{wiener} shows that it can be
reduced to the banded matrix pencil in the right hand side. Here $C$ is a
matrix whose size and entries are independent of $n$ but depend only on the
symbol $f(z)$. For more information on factorizations of the type
\eqref{wiener} see e.g.\ \cite[Prop.~2.12]{BS2} and also \cite{Day1,Hoholdt}.

From \eqref{wiener} we immediately deduce that
\begin{eqnarray}\nonumber
P_{k,n}(\lam) &:=& \det T_n\left(z^{-k}(f(z)-\lam)\right) \\
\label{wiener2} &=&
\frac{1}{\kappa}\det\left(T_n\left(z^{-q-k}A_{\lam}(z)\right)+\begin{pmatrix} C
& 0
\\ 0 & 0 \end{pmatrix}_{n\times n}\right),
\end{eqnarray}
where $\kappa\neq 0$ is a numerical constant, given by the product of the
diagonal entries of the two triangular factors $L_n$ and $R_n$ in
\eqref{wiener}.

We are now ready for the proof of Proposition \ref{prop:combinatorial2}. The
proof will follow by expanding the determinant in \eqref{wiener2} by a basic
combinatorial argument, see also \cite[Proof of Prop.~2.5]{DK}.

\begin{proof}[Proof of Proposition \ref{prop:combinatorial2}(a)] The
proposition is obvious if $m_{1,k}=0$. So we will assume below that
$m_{1,k}>0$, or equivalently \begin{equation}\label{combproof0} q+k<
k_1.\end{equation}

First we consider the case where $\lam_1\neq\infty$. By expanding the
determinant in \eqref{wiener2} we find
\begin{equation}\nonumber
P_{k,n}(\lam) = \frac{1}{\kappa}\sum_{\pi\in
S_n}\prod_{j=1}^{n}\left(a_{j-\pi(j)+k}-\lam b_{j-\pi(j)+k} +\chi_{j\leq
|C|}\chi_{\pi(j)\leq |C|}\ c_{j,\pi(j)}\right).
\end{equation}
Here $S_n$ denotes the set of all permutations of $\{1,\ldots,n\}$, and we
denote with $|C|$ the maximum of the row and column sizes of the matrix $C$;
note that this number is independent of $n$. By the band structure it follows
that we only have non-zero contributions for the permutations $\pi$ that
satisfy
\begin{equation}\label{combproof1} -q-k\leq j-\pi(j) \leq p-k\qquad\textrm{for
all $j=|C|+1,\ldots,n$}.\end{equation} Denote, for $\pi\in S_n$,
\begin{equation}\label{combproof2} N_{\pi} = \{j \mid j-\pi(j)\in\{-q-k,\ldots,-q-k+k_1-1\}\}.
\end{equation}
The set $N_{\pi}$ contains all indices $j$ for which the $(j,\pi(j))$ entry
lies in the union of the $k_1$ topmost bands of the banded matrix in
\eqref{wiener2}. By assumption \eqref{combproof0} these bands include the main
diagonal $j-\pi(j)=0$ and by definition of $k_1$ we have that the entries in
these bands are all divisible by $(\lam-\lam_1)$.

Denote the number of elements of $N_{\pi}$ in \eqref{combproof2} by
$|N_{\pi}|$. Then obviously \begin{equation}\label{combproof3} P_{k,n}\textrm{
is divisible by  }(\lam-\lam_1)^{\min_{\pi\in S_n} |N_{\pi}|},
\end{equation}
where we minimize over all permutations $\pi\in S_n$ satisfying
\eqref{combproof1}.

Let $\pi\in S_n$ satisfy \eqref{combproof1}. We give a lower bound for
$|N_{\pi}|$. Since $\sum_{j=1}^{n} (j-\pi(j))=0$ we obtain
\begin{equation}\label{combproof4}\sum_{j=1}^{n} (j-\pi(j))_+ = \sum_{j=1}^{n} (\pi(j)-j)_+,\end{equation}
where $(\cdot)_+$ is defined as $(a)_+=\max(0,a)$ for $a\in\er$. From the above
definitions we also have that
\begin{equation}\label{combproof4bis}
\left\{\begin{array}{ll}
j-\pi(j)\geq -q-k, & \textrm{ if  } j\in N_{\pi} \\
j-\pi(j)\geq -q-k+k_1, & \textrm{ if  } j\in \{|C|+1,\ldots,n\}\setminus
N_{\pi}.
\end{array}\right.
\end{equation}
By combining \eqref{combproof4} and \eqref{combproof4bis} we find that
$$ (-q-k+k_1)(n-|N_{\pi}|)+\tilde c \leq \sum_{j=1}^{n} (j-\pi(j))_+ = \sum_{j=1}^{n} (\pi(j)-j)_+
\leq (q+k)|N_{\pi}|-\tilde c.
$$
Here $\tilde c\geq 0$ is a correction term which is due to the presence of the
matrix $C$ in the top left matrix corner in \eqref{wiener2}; the number $\tilde
c$ is clearly bounded from above. We then obtain
\begin{equation}\label{combproof5} |N_{\pi}| \geq \frac{k_1-q-k}{k_1}n-c =
\left(1-\frac{q+k}{k_1}\right)n-c = m_{1,k}n-c,
\end{equation}
where we used \eqref{index:m1} and \eqref{combproof0}, and where we put
$c:=2\tilde c/k_1$. The first statement in
Proposition~\ref{prop:combinatorial2}(a) now follows from \eqref{combproof3}
and \eqref{combproof5}.

The proof of the second statement in Proposition~\ref{prop:combinatorial2}(a)
(for $\lam_1=\infty$) is similar to the one above. Now one uses that all the
entries $a_{j-\pi(j)+k}-\lam b_{j-\pi(j)+k}$ in the bands indexed by $j\in
N_{\pi}$ have their $\lam$-coefficient $b_{j-\pi(j)+k}=0$, which then yields in
a similar way to \eqref{combproof3} and \eqref{combproof5} that
\begin{equation}\label{combproof6} \deg P_{k,n}\leq n-\min_{\pi\in S_n} |N_{\pi}|
\leq (1-m_{1,k})n+c,
\end{equation}
as desired.
\end{proof}

\begin{proof}[Proof of Proposition \ref{prop:combinatorial2}(b)] Similar to Part
(a). \end{proof}

\begin{proof}[Proof of Proposition \ref{prop:combinatorial2}(c)] Part (c) follows immediately from
Parts (a) and (b), together with \eqref{index:m}, in case where
$\lam_1\neq\lam_2$. The case where $\lam_1=\lam_2$ can be obtained as well, by
observing that at least one of the numbers $m_{1,k}$ and $m_{2,k}$ must be zero
in that case. The latter follows since otherwise the numerator and denominator
in \eqref{algebraiceq:0} are equal up to multiplication with a scalar, contrary
to our assumptions.
\end{proof}

\subsection{Proof of Theorem \ref{theorem:geneigrat}}

To prove Theorem \ref{theorem:geneigrat} we need to manipulate the polynomial
$P_{k,n}(\lam)$. To this end we will use a determinant identity by K.M. Day
which we state next.

To state the identity, we need some notations. Denote with $\beta_i$ and
$\gamma_i$ the zeros of $B_1(z)$ and $B_2(z)$, respectively. Recall the
notation $z_i = z_i(\lam)$ for the roots of $A_{\lam}(z)$. Thus
\begin{eqnarray}
A_{\lam}(z) &=& c\prod_{i=1}^{p+q} (z-z_i(\lam)) \\
B_1(z) &=& c_1\prod_{i=1}^{q} (z-\beta_i)\\
B_2(z) &=& c_2\prod_{i=1}^{\deg B_2(z)} (z-\gamma_i),
\end{eqnarray}
where $c$, $c_1$, $c_2$ are non-zero constants.

The following theorem was proved under some additional hypotheses by K.M. Day
\cite{Day1}. Other proofs are in \cite{BS1,Hoholdt}, the former of them stated
under the weakest assumptions. We state the theorem in the form that is most
convenient for our purposes.

\begin{theorem} \label{theorem:Day} (Day's determinant identity). Let $k\in\{-q+1,\ldots,p-1\}$ and let
$\lam\in\cee\setminus\{\lam_1,\lam_2\}$ be such
that all roots of $A_{\lam}(z)$ are distinct. Then
\begin{equation}\label{Day:1}  P_{k,n}(\lam) = \det
T_n(z^{-k}(f-\lam)) = \sum_{S} C_{S}(\lam)(w_S(\lam))^n,
\end{equation}
where the sum is over all subsets $S\subset\{1,2,\ldots,p+q\}$ of cardinality
$|S| = q+k$ and for each such $S$ we have
\begin{equation}\label{Day:2}
w_{S}(\lam) :=
(-1)^{q+k}(a_{-q}-b_{-q}\lam)\left(\prod_{j\in S}z_j(\lam)\right)^{-1}
\end{equation}
and (with $\overline{S}:=\{1,2,\ldots,p+q\}\setminus S$)
$$ C_{S}(\lam):= \prod_{j\in \overline{S}}z_j(\lam)^{k}\prod_{\small{\begin{array}{l}i\in S,r\in R,\\
j\in \bar{S}, t\in T\end{array}}}
\frac{(z_j(\lam)-\beta_r)(\gamma_t-z_i(\lam))}{(z_j(\lam)-z_i(\lam))(\gamma_t-\beta_r)},
$$
with $R=\{1,\ldots,q\}$ and $T:=\{1,\ldots,\deg B_2(z)\}$.
\end{theorem}

Incidentally, observe that \eqref{Day:2} can be written alternatively as
\begin{equation}\label{Day:3}
w_{S}(\lam) = (-1)^{p-k}(a_p-b_p\lam)\prod_{j\in
\overline{S}}z_j(\lam).\end{equation}

We note that in case where $k\in\{0,\ldots,p-1\}$, our formulation of Theorem
\ref{theorem:Day} follows directly from the one of \cite{Day1}; for the case
$k\in\{-q+1,\ldots,-1\}$ it can be obtained from the result of \cite{Day1} by
working with the transposed matrix.

From \eqref{Day:1}--\eqref{Day:2} we see that for large $n$, the main
contribution in \eqref{Day:1} comes from those subsets $S$ for which
$|w_S(\lam)|$ is the largest possible. For $\lam\in\cee\setminus\Gamma_k$ there
is a unique such $S$, namely
$$ S=S_k:=\{1,2,\ldots,q+k\}.
$$

Now we are ready to show that the asymptotic distribution of the $k$th
generalized eigenvalues of $T_n(f)$ is described by the measure $\mu_k$,
together with possible point masses at $\lam_1$ and $\lam_2$. First we prove
this at the level of the Cauchy transforms.

\begin{proposition}\label{prop:geneigrat}
Let $k\in\{-q+1,\ldots,p-1\}$. Then
\begin{equation}\label{cauchytransform:geneig}
\lim_{n\to\infty} \int_{\cee} \frac{d\mu_{k,n}(x)}{\lam-x}= \int_{\cee}
\frac{d\mu_{k}(x)}{\lam-x} + \chi_{\lam_1\neq\infty}
\frac{m_{1,k}}{\lam-\lam_1} + \chi_{\lam_2\neq\infty}
\frac{m_{2,k}}{\lam-\lam_2}
\end{equation}
uniformly on compact subsets of $\cee\setminus\Gamma_k$.
\end{proposition}

\begin{remark}
The above proposition implicitly assumes that $\lam\neq\lam_j$, $j=1,2$.
However one checks that if $\lam_j\in\cee\setminus\Gamma_k$ then $\lam_j$ is a
removable singularity for the right hand side of
\eqref{cauchytransform:geneig}, due to the continuity of the left hand side,
and then the uniform convergence still applies.\end{remark}

\begin{proof}[Proof of Proposition~\ref{prop:geneigrat}]
As already mentioned, for large $n$ the dominant term in Day's
determinant identity Theorem~\ref{theorem:Day} is obtained by taking
$S=S_k:=\{1,2,\ldots,q+k\}$. Then we find in the same way as in \cite[Proof of
Corollary 5.3]{DK} that
\begin{multline}\lim_{n\to\infty} \int_{\cee} \frac{d\mu_{k,n}(x)}{\lam-x}
= \lim_{n\to\infty} \frac{1}{n}\sum_{\lam_i\in\textrm{sp}_k\
T_n(f)}\frac{1}{\lam-\lam_i} =  \lim_{n\to\infty}
\frac{1}{n}\frac{P'_{k,n}(\lam)}{P_{k,n}(\lam)} \\ =
\frac{w_{S_k}'(\lam)}{w_{S_k}(\lam)} \label{Day:4}  =
-\frac{w_k'(\lam)}{w_{k}(\lam)}+\chi_{\lam_1\neq\infty} \frac{1}{\lam-\lam_1}
\end{multline}
uniformly on compact subsets of $\cee\setminus\Gamma_k$, where the last
transition of \eqref{Day:4} follows from \eqref{def:wk} and \eqref{Day:2}. Now from Proposition
\ref{prop:cauchytransformcontour} we see that the right hand side of
\eqref{Day:4} equals the right hand side of \eqref{cauchytransform:geneig}. The
proposition is proved.
\end{proof}

Now we are ready for the

\begin{proof}[Proof of Theorem \ref{theorem:geneigrat}]
From the convergence of the Cauchy transforms in
Proposition~\ref{prop:geneigrat} we deduce that
$$\mu_{k,n}\ \to \ \mu_{k}+m_{1,k}\chi_{\lam_1\neq\infty}\delta_{\lam_1}+m_{2,k}\chi_{\lam_2\neq\infty}\delta_{\lam_2}$$
in the weak-star sense, which means that \eqref{weakconvergence} holds for
every continuous $\phi$ that vanishes at infinity. Now a priori, it is not
immediate that \eqref{weakconvergence} holds for all bounded continuous
functions since it is possible that $\mu_{k,n}$ has mass leaking to infinity as
$n\to\infty$. However, from Proposition~\ref{prop:combinatorial2} it follows
that this cannot happen, i.e., the measures $\{\mu_{k,n}\}_{n}$ are
\emph{tight}. Thus \eqref{weakconvergence} holds indeed for all bounded
continuous functions. For more details see \cite[Proof of Theorem 2.6]{DK}.
\end{proof}

\section{Example}\label{section:examples}

Consider the rationally generated Toeplitz matrix with symbol
\begin{equation}
\label{example:symbol1} f(z) = \frac{1}{2z^2-5z+2} = \frac{1}{(2z-1)(z-2)}
\end{equation}
defined on the complex unit circle. We may compute the Fourier series of this
symbol explicitly and find
$$ f(z) = \ldots
-\frac{1}{12z^3}-\frac{1}{6z^2}-\frac{1}{3z}-\frac{1}{6}-\frac{z}{12}-\frac{z^2}{24}-\frac{z^3}{48}\ldots.
$$
So the rationally generated Toeplitz matrix $T_n(f)$ looks like
$$ T_n(f) =
-\begin{pmatrix}
1/6 & 1/3 & 1/6 & 1/12 & \ldots \\
1/12 & 1/6 & 1/3 & 1/6 & \ldots \\
1/24 & 1/12 & 1/6 & 1/3 & \ldots \\
1/48 & 1/24 & 1/12 & 1/6 & \ldots \\
\vdots & \vdots & \vdots & \vdots & \ddots
\end{pmatrix}_{n\times n}.
$$
Equation \eqref{def:flambda} now becomes $$A_{\lam}(z) = 1-\lam(2z^2-5z+2),$$
and \eqref{def:pq} leads to $p=q=1$. The roots of $A_{\lam}(z)$ are given by
$$ z_{1,2}(\lam) = \frac{1}{4\lam} (5\lam\pm\sqrt{9\lam^2+8\lam}),
$$
and they should be labeled in such a way that $|z_1(\lam)|\leq |z_2(\lam)|$ for
all $\lam$. The roots $z_1(\lam)$ and $z_2(\lam)$ are coalescing precisely when
$9\lam^2+8\lam=0$, so the branch points are $\lam=0$ and $\lam=-8/9$.

Since $p=q=1$, there is only one relevant index $k$ in \eqref{def:curves},
namely $k=0$. The corresponding set $\Gamma_0$ is simply the line segment
connecting the branch points $\lam=0$ and $\lam=-8/9$:
$$\Gamma_0 = \{\lam\in\cee\mid |z_{1}(\lam)| = |z_{2}(\lam)|
\} = [-8/9,0].$$ This may be checked from a straightforward calculation.

Definitions \ref{def:lambda12} and \ref{def:mk} now specialize as follows:
$\lam_1=1/2$, $\lam_2=0$, $k_1=1$, $k_2=2$, and $m_{1,0}=0$, $m_{2,0}=1/2$ and
$m_{0}=1/2$. Thus the limiting eigenvalue distribution of the matrix $T_n(f)$
for $n\to\infty$ consists of an absolutely continuous part $\mu_0$ with total
mass $1/2$, supported on $\Gamma_0 = [-8/9,0]$, and a singular part which is a
point mass of mass $1/2$ at $\lam=0$.

The energy functional \eqref{energyfunctional} now specializes to
\begin{equation} \label{energy:example1} I(\nu_0) - \int
\log\frac{1}{|x-1/2|}\ d\nu_0(x).
\end{equation}
So $\mu_0$ is the minimizer of \eqref{energy:example1} over all measures $\nu_0$ on $\Gamma_0 = [-8/9,0]$
with total mass $1/2$. The
second term in \eqref{energy:example1} can be interpreted as an attraction of
$\mu_0$ towards the point $\lam=1/2$.

The measure $\mu_0$ is absolutely continuous with density given by
\eqref{def:measuresk} (with $k=0$ and $p=q=1$).  The density can be explicitly computed,
but we will omit the result since it does not lead to considerable insight. We only mention
that the density blows up like an inverse square root near both endpoints
$\lam=0$ and $\lam=-8/9$. More precisely, it behaves approximately like
$0.28/\sqrt{|\lam|}$ near $\lam=0$ and like $0.10/\sqrt{\lam+8/9}$ near
$\lam=-8/9$.

Figure~\ref{fig:density1} contains a plot of the limiting density. The figure
 shows that there is more mass near $0$ than near $-8/9$, which is due to
the attraction towards $\lam=1/2$ in \eqref{energy:example1}.

Figure~\ref{fig:density2} shows the result of a numerical computation of the
eigenvalues of $T_n(f)$ with $n=60$. Note that approximately half of the
eigenvalues is located at zero, according to
Proposition~\ref{prop:combinatorial2}; in fact we have $c=0$ in this case.\smallskip

\begin{figure}[htb]
\begin{center}
\includegraphics[scale=0.35,angle=270]{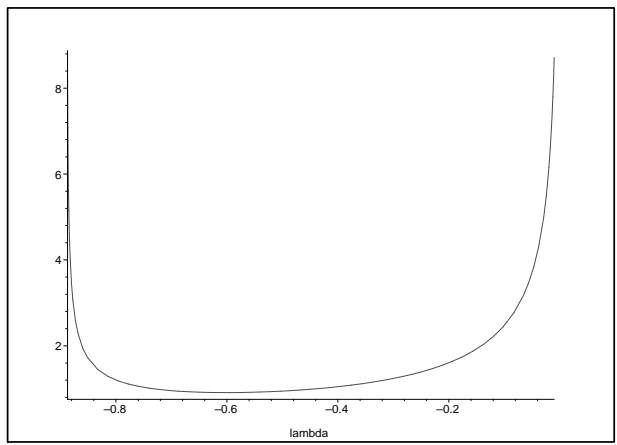}
\end{center}
\caption{Density of the measure $\mu_0$ on $\Gamma_0 = [-8/9,0]$ for the symbol
\eqref{example:symbol1}. The density blows up like a square root near both
endpoints $-8/9$ and $0$. There is more mass near $0$ due to the attraction
towards $\lam=1/2$.} \label{fig:density1}
\end{figure}

\begin{figure}[htb]
\begin{center}
\includegraphics[scale=0.35,angle=270]{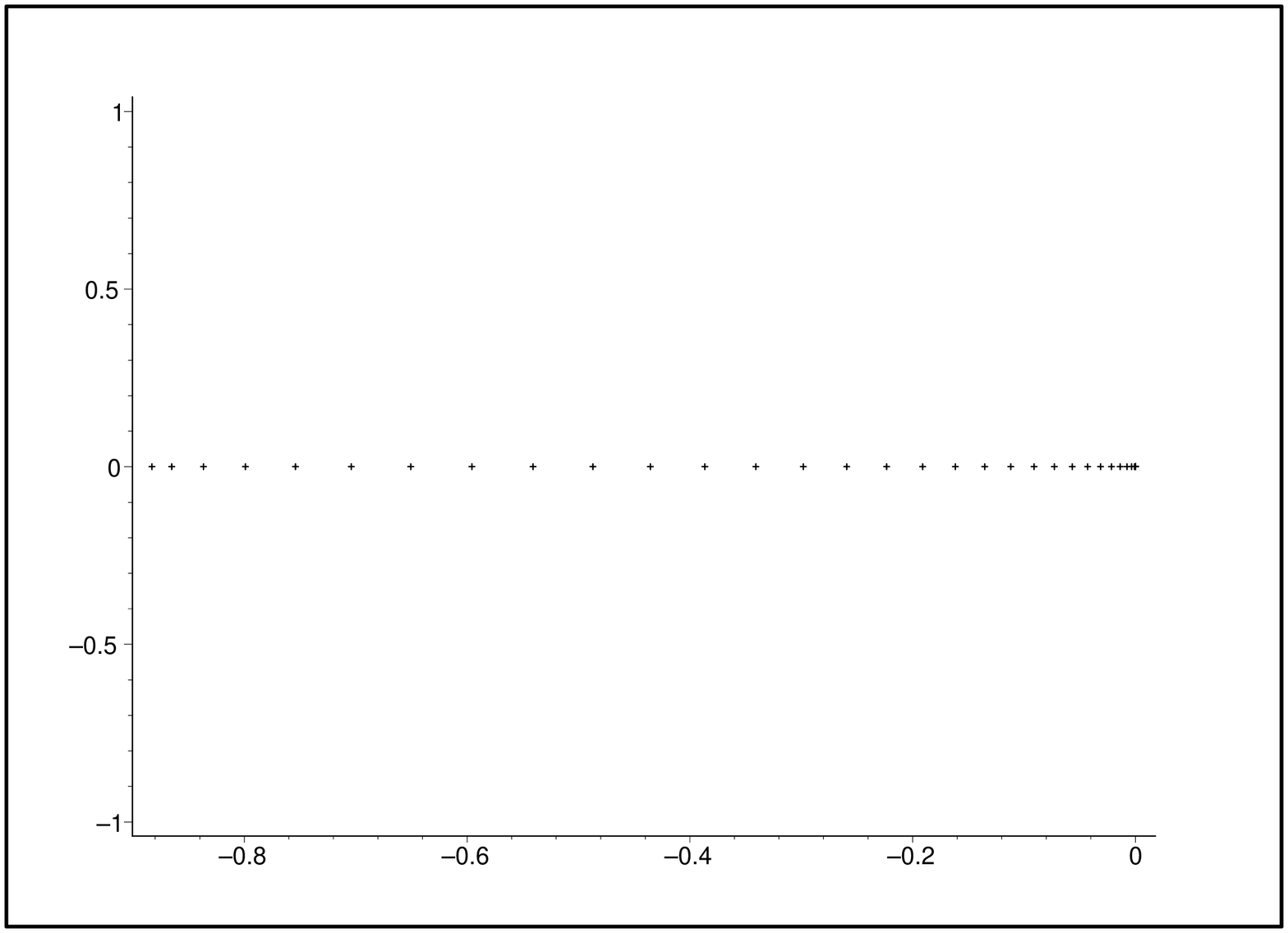}
\end{center}
\caption{Eigenvalues of the matrix $T_n(f)$ for the symbol
\eqref{example:symbol1}, computed numerically in Maple for $n=60$ using high
precision arithmetic. All the eigenvalues are real. There are $30$ of them in
the open interval $(-8/9,0)$, together with a $30$-fold eigenvalue at
$\lam=0$.} \label{fig:density2}
\end{figure}


We may consider the following modification of \eqref{example:symbol1},
\begin{equation}
\label{example:symbol2} f(z) = \frac{1+\epsilon z}{2z^2-5z+2} =
\frac{1+\epsilon z}{(2z-1)(z-2)},
\end{equation}
where $\epsilon>0$ is some small number. It is still true that $\lam_1=1/2$ and $\lam_2=0$ for any $\epsilon$,
but for $\epsilon$ non-zero we now have $k_1=k_2=1$,
$m_{1,0}=m_{2,0}=0$ and $m_{0}=1$. Thus
the limiting eigenvalue distribution of $T_n(f)$ is absolutely continuous
(without point mass), it has total mass $1$, and it is supported on the interval
$\Gamma_0$  
joining the two branch points
\begin{equation}\label{example:birthpointmass1} 
\frac{-4-5\epsilon\pm 2\sqrt{4+10\epsilon+4\epsilon^2}}{9}.\end{equation}

From the above discussions, we see that the limiting eigenvalue distribution of $T_n(f)$ is
absolutely continuous if $\epsilon>0$ and has
a point mass at the origin if $\epsilon=0$.
To understand this, note that for $\epsilon>0$ the energy functional \eqref{energyfunctional}
contains attracting point charges at both
$\lam_1=1/2$ and $\lam_2=0$ (since $k_1=k_2=1$).
In the limit $\epsilon\to 0$, the rightmost endpoint of $\Gamma_0$
in \eqref{example:birthpointmass1} moves towards the point source at $\lam_2=0$.
This causes an increasing accumulation of mass near this endpoint
which in the limit for $\epsilon=0$ gives birth to the point mass.


\section*{Acknowledgment} The authors thank professor Arno Kuijlaars for
stimulating discussions.

\end{document}